\documentclass[11pt]{amsart}
\usepackage{amsmath,amsfonts,amssymb,amsthm}
\usepackage{mathrsfs,geometry}
\usepackage[latin1]{inputenc}
\usepackage{url}
\usepackage{graphicx}

\newcommand{\iy}{\ensuremath{\infty}}
\newcommand{\R}{\ensuremath{\mathbf{R}}}

\newcommand{\N}{\ensuremath{\mathbf{N}}}

\newcommand{\E}{\ensuremath{\mathbf{E}}}

\renewcommand{\P}{\ensuremath{\mathbf{P}}}
\newcommand{\e}{\varepsilon}
\newcommand{\st}{\textnormal{ s.t. }}
\renewcommand{\leq}{\leqslant}
\renewcommand{\geq}{\geqslant}

\DeclareMathOperator{\Tr}{tr}
\newcommand{\coo}{c_{00}}

\newtheorem{theo}{Theorem}[section]

\newtheorem{pr}[theo]{Proposition}
\newtheorem{lemma}[theo]{Lemma}
\newtheorem{rk}[theo]{Remark}

\newtheorem*{theo*}{Theorem}

\title{The multiplicative property characterizes $\ell_p$ and $L_p$ norms}

\author{Guillaume Aubrun}
\address{Universit{\'e} de Lyon, Universit{\'e} Lyon 1, CNRS, UMR 5208 Institut Camille Jordan, Batiment du Doyen Jean Braconnier, 43, boulevard du 11 novembre 1918, F - 69622 Villeurbanne Cedex, France}
\email{aubrun@math.univ-lyon1.fr}

\author{Ion Nechita}
\address{Laboratoire de Physique Th{\'e}orique du CNRS, IRSAMC, Universit{\'e} de Toulouse, UPS, F-31062 Toulouse, France}
\email{nechita@irsamc.ups-tlse.fr}

\begin{document}

\begin{abstract}We show that $\ell_p$ norms are characterized as the unique norms which are both invariant under coordinate permutation and multiplicative with respect to tensor products. Similarly, the $L_p$ norms are the unique rearrangement-invariant norms on a probability space such that $\|X Y\|=\|X\|\cdot\|Y\|$ for every pair $X,Y$ of independent random variables. Our proof relies on Cram\'er's large deviation theorem.
\end{abstract}

\maketitle

\section{Introduction}

The $\ell_p$ and $L_p$ spaces are among the most important examples of Banach spaces and have been widely investigated (see e.g. \cite{handbook-ellp} for a survey). In this note, we show a new characterization of the $\ell_p/L_p$ norms by a simple algebraic identity: the {\itshape multiplicative property}. In the case of $\ell_p$ norms, this property reads as $\|x \otimes y\|=\|x\| \cdot \|y\|$ for every (finite) sequences $x,y$. In the case of $L_p$ norms, it becomes $\|XY\|=\|X\|\cdot\|Y\|$ whenever $X,Y$ are independent random variables.

Inspiration for the present note comes from quantum information theory, where the multiplicative property of the commutative and noncommutative $\ell_p$ norms plays an important role; see \cite{kuperberg,an2} and references therein.

\subsection{Discrete case: characterization of $\ell_p$ norms}

Let $\coo$ be the space of finitely supported real sequences. The coordinates of an element $x \in \coo$ are denoted $(x_i)_{i \in \N^*}$.
If $x,y \in \coo$, we define $x \otimes y$ to be double-indexed sequence $(x_iy_j)_{(i,j) \in \N^* \times \N^*}$. Throughout the paper, we consider $x \otimes y$ as an element of $\coo$ via some fixed bijective map between $\N^*$ and $\N^* \times \N^*$.

We consider a norm $\|\cdot\|$ on $\coo$ satisfying the following conditions
\begin{enumerate}
 \item {\bfseries (permutation-invariance)} If $x,y \in \coo$ are equal up to permutation of their coordinates, then $\|x\|=\|y\|$.
 \item {\bfseries (multiplicativity)} If $x,y \in \coo$, then $\|x \otimes y\|=\|x\| \cdot \|y\|$.
\end{enumerate}

Because of the invariance under permutation, the specific choice of a bijection between $\N^*$ and $\N^* \times \N^*$ is irrelevant. Examples
of a norm satisfying both conditions are given by $\ell_p$ norms, defined by
\[ \|x\|_p = \left( \sum_{i \in \N^*} |x_i|^p \right)^{1/p} \ \ \textnormal{ if }1 \leq p < +\iy \ ; \ \ \ \ \
 \|x\|_{\iy} = \sup_{i \in \N^*} |x_i| .\]

The next theorem shows that there are no other examples.

\begin{theo}\label{thm:main}
If a norm $\|\cdot\|$ on $\coo$ is permutation-invariant and multiplicative, then it coincides with $\|\cdot\|_p$ for some $p \in [1,+\iy]$.
\end{theo}

The proof of Theorem \ref{thm:main} is simple and goes as follows. First, the value of $p$ is retrieved by looking at $\|(1,1)\|$. Then, for every $x \in \coo$, the quantity $\|x\|$ is shown to equal $\|x\|_p$ by examining the statistical distribution of large coordinates of the $n$-th tensor power $x^{\otimes n}$ ($n$ large) through Cram\'er's large deviations theorem. We defer the proof to section \ref{sec:proof-discrete}.

\subsection{Continuous case: characterization of $L_p$ norms}

We now formulate a version of Theorem \ref{thm:main} in a continuous setting, in order to characterize $L_p$ norms. Let $(\Omega,\mathcal{F},\mathbf{P})$ be a {\itshape rich} probability space, which means that it is possible to define on it one continuous random variable. This implies that we can define on $\Omega$ an arbitrary number of independent random variables with arbitrary distributions; one can think of $\Omega$ as the interval $[0,1]$ equipped with the Lebesgue measure. A random variable is said to be \emph{simple} if it takes only finitely many values. For a random variable $X : \Omega \to \R$, the $L_p$ normd are defined as
\[ \|X\|_{L_p} = \begin{cases} \left( \E |X|^p \right)^{1/p} & \textnormal{if } 1 \leq p < +\iy, \\ \inf \{ M \st \mathbf{P}(|X|\leq M)=1 \} & \textnormal{if } p = \iy . \end{cases} \]

The $L_p$ norms are rearrangement-invariant (i.e. the norm of a random variable depends only on its distribution) and satisfy the property $\|XY\| = \|X\|\cdot\|Y\|$ whenever $X,Y$ are independent random variables. Note that product of independant random variables correspond to the tensor product in $\coo$. These properties characterize the $L_p$ norms:

\begin{theo} \label{thm:continuous}
Let $(\Omega,\mathcal{F},\mathbf{P})$ be a rich probability space, and let $\mathcal{E}$ be the space of simple random variables. Let $\|\cdot\|$ be a norm on $\mathcal{E}$ with the following properties :
\begin{enumerate}
\item If two random variables $X,Y \in \mathcal{E}$ have the same distribution, then $\|X\|=\|Y\|$,
\item If two random variables $X,Y \in \mathcal{E}$ are independent, then $\|XY\|=\|X\|\cdot\|Y\|$.
\end{enumerate}
Then there exists $p \in [1,+\iy]$ such that $\|X\|=\|X\|_{L_p}$ for every $X \in \mathcal{E}$.
\end{theo}

We prove Theorem \ref{thm:continuous} in Section \ref{sec:cont}. We will derive Theorem \ref{thm:continuous} as a consequence of Theorem \ref{thm:main}. Alternatively one could prove it by mimicking the proof of Theorem \ref{thm:main}.

\section{The case of $\ell_p$ norms: proof of Theorem \ref{thm:main}} \label{sec:proof-discrete}

Let $\|.\|$ be a norm on $\coo$ which is permutation-invariant and multiplicative.

\subsection*{STEP 1} We first show that the norm of an element of $\coo$ depends only on the absolute values of its coordinates.

\begin{lemma}\label{lem:unconditional}
A norm on $\coo$ which is permutation-invariant and multiplicative is also {\bf unconditional}: if $x,y \in \coo$ have coordinates
with equal absolute values ($|x_i|=|y_i|$ for every $i$), then $\|x\|=\|y\|$. As a consequence, if $a,b \in \coo$ and
$0 \leq a \leq b$ (coordinatewise), then $\|a\| \leq \|b\|$.
\end{lemma}

\begin{proof}
If $x,y$ have coordinates with equal absolute values, then the vectors $x \otimes (1,-1)$ and $y \otimes (1,-1)$ are equal up to
permutation of their coordinates. Permutation-invariance and multiplicativity imply that $\|x\|=\|y\|$. For the second part of the lemma, note that $0 \leq a \leq b$ implies that $a$ belongs to the convex hull of the vectors $\{ (\e_i b_i) ; \e_i = \pm 1 \}$ and use the triangle inequality to conclude.
\end{proof}

\begin{rk}
In the literature, unconditional and permutation-invariant norms are sometimes called \emph{symmetric norms}.
\end{rk}

\subsection*{STEP 2} We now focus on sequences whose nonzero coefficients are equal to $1$. We write ${\bf 1}^n$ for the sequence formed with $n$ 1's followed by infinitely many zeros and we put $u_n = \|{\bf 1}^n\|$. By Lemma \ref{lem:unconditional}, the sequence $(u_n)_n$ is non-decreasing. Moreover, the multiplicativity property of the norm implies that the sequence $(u_n)_n$ itself is multiplicative: $u_{kn}=u_k u_n$. It is folklore that a nonzero non-decreasing sequence $(u_n)_n$ such that $u_{kn}=u_ku_n$ must equal $(n^{\alpha})_n$ for some $\alpha \geq 0$ (see \cite{howe} for a proof). We set $p=1/\alpha$ ($p = +\iy$ if $\alpha=0$). By the triangle inequality, $u_{n+k} \leq u_n + u_k$, which implies that $p \geq 1$. At this point we have proved that
\[ \| {\bf 1}^n \| = n^{1/p}. \]

To prove Theorem \ref{thm:main}, we need to show that $\|x\|=\|x\|_p$ for every $x \in \coo$. The case $p=+\iy$ is easily handled, so we may assume that $1 \leq p < +\iy$. By Lemma \ref{lem:unconditional}, without loss of generality, we may also assume that the coordinates of $x$ are non-negative and in non-increasing order. Let $k$ be the number of nonzero coordinates of $x$ ; then $x_i=0$ for $i>k$. We will separately show the inequalities $\|x\| \geq \|x\|_p$ and $\|x\| \leq \|x\|_p$. In both cases, we compare $x^{\otimes n}$ with simpler vectors and apply Cram\'er's theorem (which we now review) to estimate the number of ``large'' coordinates of $x^{\otimes n}$ when $n$ goes to infinity.

\subsection*{Cram\'er's theorem}

Fix $x \in \coo$ with non-negative non-increasing coordinates, and let $k$ be the number of nonzero coordinates of $x$. For $a>0$, let $N(x,a)$ be the number of coordinates of $x$ which are larger than or equal to $a$. To estimate this number, we introduce the convex function $\Lambda_x : \R \to \R$
\[ \Lambda_x(\lambda) = \ln \left( \sum_{i=1}^k x_i^\lambda \right) \]
and its convex conjugate $\Lambda_x^* : \R \to \R \cup \{+\iy\}$
\[ \Lambda_x^*(t) = \sup_{\lambda \in \R} \lambda t -\Lambda_x(\lambda) .\]
The Fenchel--Moreau theorem (see e.g. \cite{brezis}) implies that convex conjugation is an involution:  we have,  for any $\lambda \in \R$,
\[ \Lambda_x(\lambda) = \sup_{t \in \R} \lambda t - \Lambda_x^*(t). \]

\begin{pr}[Cram\'er's large deviation theorem] \label{prop:Cramer}
Let $x \in \coo$ such that $x_i >0$ for $1 \leq i \leq k$ and $x_i=0$ for $i>k$. Let $t$ be a real number such that
$\exp(t) \leq \|x\|_{\iy}$. Then,
\[ \lim_{n \to \iy} \frac{1}{n} \ln N(x^{\otimes n}, \exp(tn)) =  \left.\begin{cases} \ln k & \textnormal{if } \exp(t) \leq (\prod_{i=1}^k x_i)^{1/k}  \\  -\Lambda_x^*(t) & \textnormal{otherwise} \end{cases}\right\}  \geq -\Lambda_x^*(t).\]
\end{pr}

\begin{proof}
To see how Proposition \ref{prop:Cramer} follows from the standard formulation of Cram\'er's theorem, let $(X_n)$ be independent random variables with common distribution given by
\[ \frac{1}{k} \sum_{i=1}^k \delta_{\ln x_i} . \]
Then $\P( \frac{1}{n} ( X_1 +\dots + X_n ) \geq t) = \frac{1}{k^n} N(x^{\otimes n},\exp(tn))$. The usual Cram\'er theorem (see any probability textbook, or \cite{cerf-petit} for a short proof)  asserts that
\[ \lim_{n \to \iy} \frac{1}{n} \ln \P \left( \frac{1}{n} ( X_1 +\dots + X_n ) \geq t\right) = \begin{cases} 0 & \textnormal{if } t \leq \E X_1  \\  \displaystyle -\sup_{\lambda \in \R} \left( \lambda t - \ln \E e^{\lambda X_1} \right)& \textnormal{otherwise. } \end{cases} \]
This is equivalent to the equality in Proposition \ref{prop:Cramer}. The last inequality follows easily since $\Lambda_x^*(t) \geq -\ln k$ for every real $t$.
\end{proof}

We now complete the proof of the main theorem by comparing $x^{\otimes n}$ with simpler vectors, as shown in Figure \ref{fig:x-otimes-n}.

\begin{figure}[htbp]
\centering
\includegraphics{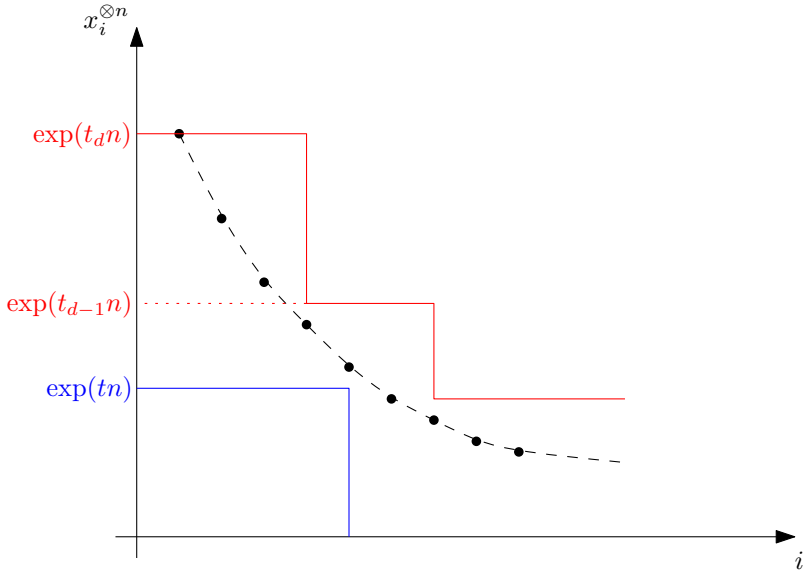}
\caption{Bounding the vector $x^{\otimes n}$ by vectors with simpler profiles. The coordinates of the tensor power $x^{\otimes n}$ are represented by dark circles, the vector used in for the lower bound has only one non-zero value $\exp(tn)$ and the upper-bounding vector has values $\exp(t_d n) \geq \cdots \geq \exp(t_1 n) \geq 0$.}
\label{fig:x-otimes-n}
\end{figure}

\subsection*{STEP 3: the lower bound $\|x\| \geq \|x\|_p$}

For $t \in \R$, we have the lower bound
\[ \|x\| = \| x^{\otimes n} \|^{1/n} \geq \| \exp(tn) {\bf 1}^{N(x^{\otimes n},\exp(tn))} \|^{1/n} = \exp(t) N(x^{\otimes n},\exp(tn))^{1/np} .\]
Proposition \ref{prop:Cramer} asserts that
\[ \lim_{n \to \iy} N(x^{\otimes n},\exp(tn))^{1/n} \geq \exp(-\Lambda_x^*(t)) .\]
We have therefore
\[ \|x\| \geq  \exp(t-\Lambda_x^*(t)/p) = \exp(pt-\Lambda_x^*(t))^{1/p} \]
for any $t \in \R$. Taking the supremum over $t$ and using the Fenchel--Moreau theorem shows that
\[ \|x\| \geq \exp(\Lambda_x(p))^{1/p} = \|x\|_p .\]

\subsection*{STEP 4: The upper bound $\|x\| \leq \|x\|_p$}

Fix $\e >0$ and choose $t_0<\dots<t_d$ such that
\[ \exp(t_0) = \min_{1 \leq i \leq k} x_k, \ \  \exp(t_1) = \left(\prod_{i=1}^k x_i\right)^{1/k}, \ \ \exp(t_d)=\|x\|_{\iy} \ \textnormal{ and } \ \sup_{2\leq i \leq d} |t_i-t_{i-1}| < \e.
\]
For $n \in \N^*$, we define a vector $y_n \in \coo$ as follows: the coordinates of $y_n$ belong to the set
\[ \{0,\exp(nt_1),\exp(nt_2),\dots,\exp(nt_d)\} \]
and are minimal possible such that the inequality $x^{\otimes n} \leq y_n$ holds coordinatewise. Lemma \ref{lem:unconditional} implies that $\|x^{\otimes n}\| \leq \| y_n\|$. On the other hand, for $1 \leq i \leq d$, the number of coordinates of $y_n$ which are equal to $\exp(nt_i)$ is less than $N(x^{\otimes n}, \exp(nt_{i-1}))$. The triangle inequality implies that
\begin{eqnarray*}
\| y_n \| &\leq& \sum_{i=1}^d \left\| \exp(t_in) {\bf 1}^{N(x^{\otimes n},\exp(t_{i-1}n))} \right\| \\
& \leq & \sum_{i=1}^d \exp(t_in) N(x^{\otimes n},\exp(t_{i-1}))^{1/p} \\
& \leq & d \max_{1 \leq i \leq d} \left\{ \exp(t_i n) N(x^{\otimes n},\exp(t_{i-1}n))^{1/p} \right\}.
\end{eqnarray*}
This gives an upper bound for $\|x\|$
\begin{equation} \label{eq:upperbound}
  \|x\| = \|x^{\otimes n}\|^{1/n} \leq \|y_n\|^{1/n} \leq d^{1/n} \max_{1 \leq i \leq d} \left\{ \exp(t_i) N(x^{\otimes n},\exp(t_{i-1}n))^{1/np} \right\}.
\end{equation}
For $2 \leq i \leq d$, Proposition \ref{prop:Cramer} implies that
\begin{eqnarray*}
\lim_{n \to \iy} \exp(t_i) N(x^{\otimes n},\exp(t_{i-1}n))^{1/np} & = & \exp(t_i) \exp(- \Lambda_x^*(t_{i-1}) )^{1/p}\\
& \leq & \exp(t_i) \exp( -(pt_{i-1} - \Lambda_x(p)) )^{1/p} \\
& \leq & \exp(\e) \|x\|_p.
\end{eqnarray*}
Similarly, for $i=1$,
\[ \exp(t_1) N(x^{\otimes n},\exp(t_0n))^{1/np} \leq \exp(t_1) k^{1/p} \leq \|x\|_p,\]
where the last inequality follows from the inequality of arithmetic and geometric means. Therefore, taking the limit $n \to \iy$ in inequality \eqref{eq:upperbound} implies that $\|x\| \leq \exp(\e) \|x\|_p$, and the result follows when $\e$ goes to $0$.

\section{The case of $L_p$ norms: proof of Theorem \ref{thm:continuous}}\label{sec:cont}

Let $\|\cdot\|$ be a norm on the space $\mathcal{E}$ of simple random variables which satisfies the hypotheses of Theorem \ref{thm:continuous}. Throughout the proof, we denote by $B_n \in \mathcal{E}$ a Bernoulli random variable with parameter $1/n$, i.e. such that $\P(B_n=1)=1/n$ and $\P(B_n=0)=1-1/n$. Moreover, we assume that the random variables $(B_n)_{n \in \N}$ are independent.

We will define a norm $|||\cdot|||$ on $\coo$ which will satisfy the hypotheses of Theorem \ref{thm:main}. It is convenient to identify $\coo$ with the union of an increasing sequence of subspaces
\begin{equation} \label{eqn:inductive-limit} \coo = \bigcup_{n \in \N} \R^n .\end{equation}
For $x = (x_1,\dots,x_n) \in \R^n$, we define
\[ |||x||| = \frac{\|X\|}{\|B_n\|} ,\]
where $X \in \mathcal{E}$ is a random variable with distribution $\frac{1}{n} (\delta_{x_1} + \cdots + \delta_{x_n})$. 

This defines a norm on $\coo$ provided the construction is compatible with the union \eqref{eqn:inductive-limit}. To check this, consider $x$ as an element of $\R^m$ for $m>n$, obtained by padding $x$ with $m-n$ zeros. Let $X'$ be a random variable with distribution $\frac{1}{m} (\delta_{x_1} + \cdots + \delta_{x_n} + (m-n)\delta_0)$. If we moreover assume that the random variables $X,X',B_n,B_m$ are independent, it is easily checked that $XB_m$ and $X'B_n$ both have the distribution $\frac{1}{nm} (\delta_{x_1} + \cdots + \delta_{x_n})  + (1-\frac{1}{nm})\delta_0$. By the hypotheses on the norm, this implies that $\|X\|\cdot\|B_m\|= \|X'\| \cdot \|B_n\|$ and therefore
\[ \frac{\|X\|}{\|B_n\|} = \frac{\|X'\|}{\|B_m\|} .\]
This shows that $|||x|||$ is properly defined for $x \in \coo$. It is easily checked that $|||\cdot|||$ is a norm on $\coo$ which is both permutation-invariant and multiplicative (for the latter, use the fact that $B_nB_m$ and $B_{nm}$ have the same distribution). 

By Theorem \ref{thm:main}, the norm $|||\cdot|||$ equals the norm of $\ell_p$ for some $p \in [1,+\infty]$. To compute $\|B_n\|$, consider the vector $x \in \R^{2n}$ given by $n$ $1$'s followed by $n$ $0$'s. We have 
\[ n^{1/p} = \|x\|_p = |||x||| = \frac{\|B_{2}\|}{\|B_{2n}\|} = \frac{1}{\|B_n\|} ,\]
where the last equality follows from the aforementioned property of Bernoulli random variables. This implies that the equation
\begin{equation} \label{eqn:capital-ellp} \|X\| = \|X\|_{L_p} .\end{equation}
holds for every $X \in \mathcal{E}$ with rational weights, i.e. with distribution $\frac{1}{n}(\delta_{x_1} + \dots + \delta_{x_n})$ for some $n$. The extension to all random variables in $\mathcal{E}$ follows by an approximation argument. Indeed, for every positive random variable $X \in \mathcal{E}$, there exist sequences $(Y_n),(Z_n)$ of positive random variables, with rational weights, such that
\[ Y_n \leq X \leq Z_n \]
and
\[ \lim_{n \to \iy} \|Y_n\|_{L_p} =  \lim_{n \to \iy} \|Z_n\|_{L_p} = \|X\|_{L_p} .\]
Therefore, we may use the following lemma (a continuous version of lemma \ref{lem:unconditional}) to extend formula \eqref{eqn:capital-ellp} to every $X \in \mathcal{E}$.

\begin{lemma} \label{lem:stochastic-domination}
Let $\|\cdot\|$ be a norm on $\mathcal{E}$ which satisfies the hypotheses of Theorem \ref{thm:continuous}. If $X \in \mathcal{E}$, then the random variables $X$ and $|X|$ have the same norm. If $X,Y \in \mathcal{E}$ are two random variables such that $0 \leq X \leq Y$, then $\|X\| \leq \|Y\|$.
\end{lemma}

\begin{proof}
To prove the first part, note that if $\e$ is a random variable which is independent from $X$ and such that $\P(\e=1)=\P(\e=-1)=1/2$, then $\e X$ and $\e |X|$ are identically distributed. Assume now that $0 \leq X \leq Y$. There exists a finite measurable partition $(\Omega_1,\dots,\Omega_n)$ of $\Omega$ such that $X$ and $Y$ are constant on each set $\Omega_i$. Let $x_i$ (resp. $y_i$) be the value of $X$ (resp. $Y$) on $\Omega_i$; then $x_i \leq y_i$. For any $\e=(\e_1,\dots,\e_n) \in \{\pm 1\}^n$, one may define a random variable $Z_\e$ by setting $Z_\e(\omega)=\e_i$ for $\omega \in \Omega_i$. The random variable $X$ can be written as a convex combination of the random variables $\{Z_\e Y\}_{\e \in \{\pm 1\}^n}$ (this is a consequence of the fact that $(x_1,\dots,x_n)$ is in the convex hull of $(\pm y_1, \dots,\pm y_n)$---a fact already used in the proof of Lemma \ref{lem:unconditional}). We now conclude by the triangle inequality and the
fact that $\|Z_\e Y\|=\|Y\|$ since both variables are equal in absolute value.
\end{proof}

\section{Extensions}

\subsection{Extension to the complex case}

Theorems \ref{thm:main} and \ref{thm:continuous} extend easily to the complex case. We only state the discrete version. Up to a small detail, the proof is the same as in the real case.

\begin{theo}\label{thm:complex}
Let $\|\cdot\|$ be a permutation-invariant and multiplicative norm on the space of finitely supported complex sequences. Then, there exists some $p \in [1,+\iy]$ such that $\| \cdot \| = \| \cdot\|_p$.
\end{theo}
\begin{proof}
We argue in the same way as we did for real sequences. The proof adapts \emph{mutatis mutandis}, except for the first part of Lemma \ref{lem:unconditional} whose proof requires a slight modification. Let $\omega$ be a primitive $k$-th root of unity. If the coordinates of $x$ and $y$ differ only by a power of $\omega$, then the vectors $x \otimes (1,\omega,\dots,\omega^{k-1})$ and $y \otimes (1,\omega,\dots,\omega^{k-1})$ are equal up to permutation of coordinates, and therefore $\|x\|=\|y\|$. The case of a general complex phase follows by continuity.
\end{proof}

\subsection{Noncommutative setting}

Theorem \ref{thm:main} can be formulated to characterize the Schatten $p$-norms.

Let $H$ be a infinite-dimensional (real or complex) separable Hilbert space and $F(H)$ be the space of finite rank operators on $H$. Let $\|\cdot\|$ a norm on $F(H)$ which is {\bfseries unitarily invariant}: whenever $U,V$ are unitary operators on $H$ and $A \in F(H)$, we have
$\| UAV \| = \|A\|$. Assume also that the norm is {\bfseries multiplicative} in the following sense: for any $A,B \in F(H)$,
\[ \| A \otimes B \| = \|A\| \cdot \|B\| .\]
As in the commutative case, we fix a isometry between $H$ and the Hilbertian tensor product $H \otimes H$ to define $\|A \otimes B\|$---the particular choice we make is irrelevant because of the unitary invariance. The next theorem asserts that the only norms which are unitarily invariant and multiplicative are the Schatten $p$-norms defined as $\|A\|_p=(\Tr |A|^p)^{1/p}$ for $1 \leq p < +\iy$, while $p=\iy$ corresponds to the operator norm.

\begin{theo}\label{thm:Schatten}
Let $\|\cdot\|$ be a norm on the space of finite-rank operators on a infinite-dimensional Hilbert space which is both multiplicative and unitarily invariant. Then, there exists some $p \in [1,+\iy]$ such that $\| \cdot \|$ is the Schatten $p$-norm.
\end{theo}
\begin{proof}
By a result of von Neumann 
(see \cite{bhatia}, Theorem IV.2.1),
a norm $N$ on $F(H)$ is unitarily invariant if and only if $N(\cdot)=\|s(\cdot)\|$ for some symmetric norm $\|.\|$ on $\coo$---here $s(A) \in \coo$ denotes the list of singular values of an operator $A \in F(H)$. The result follows then from the commutative case (Theorem \ref{thm:main} or Theorem \ref{thm:complex}).
\end{proof}

\end{document}